\documentclass[11pt,a4paper]{article}

\usepackage{inputenc}
\usepackage{amsmath}
\usepackage{bm}
\usepackage{bbold}
\usepackage{amsthm}

\usepackage{hyperref}

\title{Algebraic Modelling and Performance Evaluation\\ of Acyclic Fork-Join Queueing Networks\thanks
{Advances in Stochastic Simulation Methods / Ed. by N.~Balakrishnan, V.~B.~Melas, S.~Ermakov (Statistics for Industry and Technology), Birkh\"{a}user, Boston, 2000, pp.~63--81.}} 

\author{Nikolai K. Krivulin\thanks{Faculty of Mathematics and Mechanics, St.~Petersburg State University, 28 Universitetsky Ave., St.~Petersburg, 198504, Russia, 
nkk@math.spbu.ru}
}

\date{}

\newtheorem{theorem}{Theorem}
\newtheorem{lemma}[theorem]{Lemma}

\newtheorem{proposition}{Proposition}

\begin{document}

\maketitle

\begin{abstract}
Simple lower and upper bounds on mean cycle time in stochastic acyclic
fork-join queueing networks are derived using a $(\max,+)$-algebra based
representation of network dynamics. The behaviour of the bounds under various
assumptions concerning the service times in the networks is discussed, and
related numerical examples are presented.
\\

\textit{Key-Words:} (max,+)-algebra, dynamic state equation, acyclic fork-join queueing networks, stochastic dynamic systems, mean cycle time.
\end{abstract}

\section{Introduction}

Fork-join networks introduced in \cite{Baccelli1989Queueing,Baccelli1989Acyclic}, present a class of queueing system models which allow customers (jobs,
tasks) to be split into several parts, and to be merged into one when they
circulate through the system. The fork-join formalism proves to be useful in
the description of dynamical processes in a variety of actual complex systems,
including production processes in manufacturing, transmission of messages in
communication networks, and parallel data processing in multi-processor
computer systems. As a natural illustration of the fork and join operations,
one can consider respectively splitting a message into packets in a
communication network, each intended for transmitting via separate ways, and
merging packets at a destination node of the network to restore the message.
Further examples can be found in \cite{Baccelli1989Queueing}.

The usual way to represent the dynamics of fork-join queueing networks relies 
on the implementation of recursive state equations of the Lindley type 
\cite{Baccelli1989Queueing}. Since the recursive equations associated with 
the fork-join networks can be expressed only in terms of the operations of 
maximum and addition, there is a possibility to represent the dynamics of the 
networks in terms of the (max,+)-algebra which is actually an algebraic system 
just supplied with the same two operations \cite{Cuninghame-Green1979Minimax,Baccelli1993Synchronization,Maslov1994Idempotent}. In fact, (max,+)-algebra models 
offer a more compact and unified way of describing network dynamics, and, 
moreover, lead to equations closely analogous to those in the conventional 
linear system theory \cite{Baccelli1993Synchronization,Krivulin1994Using,Krivulin1995Amax-algebra,Krivulin1996Themax-plus,Krivulin1996Max-plus}. In that case, the (max,+)-algebra approach gives one the chance to 
exploit results and numerical procedures available in the algebraic system 
theory and computational linear algebra.

One of the problems of interest in the analysis of stochastic queueing
networks is to evaluate the mean cycle time of a network. Both the mean cycle
time and its inverse which can be regarded as a throughput present performance 
measures commonly used to describe efficiency of the network operation.

It is frequently rather difficult to evaluate the mean cycle time exactly,
even though the network under study is quite simple. To get information about
the performance measure in this case, one can apply computer simulation to
produce reasonable estimates. Another approach is to derive bounds on the mean
cycle time. Specifically, a technique which allows one to establish bounds
based on results of the theory of large deviations as well as the
Perron-Frobenius spectral theory has been introduced in \cite{Baccelli1991Estimates}.

In this paper we propose an approach to get bounds on the mean cycle time,
which exploits the (max,+)-algebra representation of acyclic fork-join
network dynamics derived in \cite{Krivulin1996Themax-plus,Krivulin1996Max-plus}. This approach is
essentially based on pure algebraic manipulations combined with application of
bounds on extreme values, obtained in \cite{Gumbel1954Themaxima,Hartley1954Universal}.

The rest of the paper is organized as follows. Section~\ref{S-PAD} presents
basic (max,+)-algebra definitions and related results which underlie the
development of network models and their analysis in the subsequent sections.
In Section~\ref{S-FAR}, further algebraic results are included which provide
a basis for derivation of bounds on the mean cycle time.

A (max,+)-algebra representation of the fork-join network dynamics and related
examples are given in Section~\ref{S-AMQ}. Furthermore, Section~\ref{S-MPR}
offers some monotonicity property for the networks, which is exploited in
Section~\ref{S-BSC} to get algebraic bounds on the service cycle completion
time. Stochastic extension of the network model is introduced in
Section~\ref{S-SEN}. The section concludes with a result which provides simple
bounds on the network mean cycle time. Finally, Section~\ref{S-DEX} presents
examples of calculating bounds and related discussion.

\section{Preliminary Algebraic Definitions and Results}\label{S-PAD}

The $(\max,+)$-algebra presents an idempotent commutative semiring (idempotent
semifield) which is defined as
$ \mathbb{R}_{\max}
=\langle\underline{\mathbb{R}},\oplus,\otimes\rangle $
with $ \underline{\mathbb{R}}=\mathbb{R} \cup \{\varepsilon\} $,
$ \varepsilon=-\infty $, and binary operations $ \oplus $ and
$ \otimes $ defined as
$$
x \oplus y=\max(x,y), \quad x \otimes y=x + y, \quad
\mbox{for all} \quad x,y\in\underline{\mathbb{R}}.
$$

As it is easy to see, the operations $ \oplus $ and $ \otimes $
retain most of the properties of the ordinary addition and multiplication,
including associativity, commutativity, and distributivity of multiplication
over addition. However, the operation $ \oplus $ is idempotent; that
is, for any $ x \in \underline{\mathbb{R}} $, one has
$ x \oplus x=x $.

There are the null and identity elements in the algebra, namely
$ \varepsilon $ and $ 0 $, to satisfy the conditions
$ x\oplus\varepsilon=\varepsilon\oplus x=x $, and
$ x \otimes 0=0\otimes x=x $, for any
$ x \in \underline{\mathbb{R}} $. The null element $ \varepsilon $ and
the operation $ \otimes $ are related by the usual absorption rule
involving $ x \otimes \varepsilon=\varepsilon \otimes x=\varepsilon $.

Non-negative integer power of any $ x\in\mathbb{R} $ can be defined as
$ x^{0}=0 $, and $ x^{q}=x\otimes x^{q-1}=x^{q-1}\otimes x $ for
$ q\geq1 $. Clearly, the $(\max,+)$-algebra power $ x^{q} $ corresponds
to $ qx $ in ordinary notations. We will use the power notations only in
the $(\max,+)$-algebra sense.


The $(\max,+)$-algebra of matrices is readily introduced in the regular way.
Specifically, for any $(n\times n)$-matrices $ X=(x_{ij}) $ and
$ Y=(y_{ij}) $, the entries of $ U=X \oplus Y $ and
$ V=X \otimes Y $ are calculated as
$$
u_{ij}=x_{ij} \oplus y_{ij}, \quad \mbox{and} \quad
v_{ij}=\bigoplus_{k=1}^{n} x_{ik} \otimes y_{kj}.
$$

As the null and identity elements, the matrices
$$
{\cal E}=\left(\begin{array}{ccc}
           \varepsilon & \ldots & \varepsilon \\
           \vdots      & \ddots & \vdots \\
           \varepsilon & \ldots & \varepsilon
         \end{array}\right), \qquad
       I=\left(\begin{array}{ccc}
          0           &        & \varepsilon \\
                      & \ddots & \\
          \varepsilon &        & 0
         \end{array}\right)
$$
are respectively taken in the algebra.

The matrix operations $ \oplus $ and $ \otimes $ possess monotonicity
properties; that is, the matrix inequalities $ X\leq U $ and
$ Y\leq V $ result in
$$
X\oplus Y\leq U\oplus V, \qquad X\otimes Y\leq U\otimes V
$$
for any matrices of appropriate size.

Let $ X\ne{\cal E} $ be a square matrix. In the same way as in the
conventional algebra, one can define $ X^{0}=I $, and
$ X^{q}=X\otimes X^{q-1}=X^{q-1}\otimes X $ for any integer $ q\geq 1 $.
However, idempotency leads, in particular, to the matrix identity
$$
(X\oplus Y)^{q}=X^{q}\oplus X^{q-1}\otimes Y\oplus\cdots\oplus Y^{q}.
$$
As direct consequences of the above identity, one has
$$
(X\oplus Y)^{q} \geq X^{p}\otimes Y^{q-p}, \qquad
(I\oplus X)^{q} \geq (I\oplus X)^{p} \geq X^{p},
$$
for all $ p=0,1,\ldots,q $.

For any matrix $ X $, its norm is defined as
$$
\|X\|=\bigoplus_{i,j}x_{ij}=\max_{i,j}x_{ij}.
$$
The matrix norm possesses the usual properties. Specifically, for any matrix
$ X $, it holds $ \|X\|\geq\varepsilon $, and
$ \|X\|=\varepsilon $ if and only if $ X={\cal E} $. Furthermore, we
have $ \|c\otimes X\|=c\otimes\|X\| $ for any
$ c\in\underline{\mathbb{R}} $, as well as additive and multiplicative
properties involving
$$
\|X\oplus Y\|=  \|X\|\oplus\|Y\|, \qquad
\|X\otimes Y\| \leq \|X\|\otimes\|Y\|
$$
for any two conforming matrices $ X $ and $ Y $. Note that for any
$ c>0 $, we also have $ \|cX\|=c\|X\| $.


Consider an $(n \times n)$-matrix $ X $ with its entries
$ x_{ij} \in \underline{\mathbb{R}} $. It can be treated as an adjacency
matrix of an oriented graph with $ n $ nodes, provided each entry
$ x_{ij} \neq \varepsilon $ implies the existence of the arc
$ (i,j) $ in the graph, while $ x_{ij}=\varepsilon $ does the lack
of the arc.

It is easy to verify that for any integer $ q\geq 1 $, the matrix
$ X^{q} $ has its the entry $ x^{(q)}_{ij}\ne\varepsilon $ if
and only if there exists a path from node $ i $ to node $ j $ in
the graph, which consists of $ q $ arcs. Furthermore, if the graph
associated with the matrix $ X $ is acyclic, we have
$ X^{q}={\cal E} $ for all $ q>p $, where $ p $ is the
length of the longest path in the graph. Otherwise, provided that the graph is
not acyclic, one can construct a path of any length, lying along circuits, and
then it holds that $ X^{q}\neq{\cal E} $ for all $ q \geq 0 $.

Consider the implicit equation in an unknown vector
$ \mbox{\boldmath $x$}=(x_{1},\ldots,x_{n})^{T} $,
\begin{equation}\label{e-1}
\mbox{\boldmath $x$}=U\otimes\mbox{\boldmath $x$}\oplus\mbox{\boldmath $v$},
\end{equation}
where $ U=(u_{ij}) $ and
$ \mbox{\boldmath $v$}=(v_{1},\ldots,v_{n})^{T} $ are respectively
given $(n\times n)$-matrix and $n$-vector. Suppose that the entries of the
matrix $ U $ and the vector $ \mbox{\boldmath $v$} $ are either
positive or equal to $ \varepsilon $. It is easy to verify (see, e.g.
\cite{Cuninghame-Green1979Minimax,Cohen1985Alinear-system-theoretic} that equation (\ref{e-1}) has the
unique bounded solution if and only if the graph associated with $ U $ is
acyclic. Provided that the solution exists, it is given by
\begin{equation}\label{e-2}
\mbox{\boldmath $x$}=(I\oplus U)^{p}\otimes\mbox{\boldmath $v$},
\end{equation}
where $ p $ is the length of the longest path in the graph.

\section{Further Algebraic Results}\label{S-FAR}

We start with an obvious statement.

\begin{proposition}\label{P-XXG}
For any matrix $ X $, it holds
$$
X \le \|X\|\otimes G,
$$
where $ G $ is the adjacency ($\varepsilon$--0)-matrix of the graph
associated with $ X $.
\end{proposition}

\begin{proposition}\label{P-XEE}
Suppose that matrices $ X_{1},\ldots,X_{k} $ have a common associated
acyclic graph, $ p $ is the length of the longest path in the graph, and
$$
X=X_{1}^{m_{1}}\otimes\cdots\otimes X_{k}^{m_{k}},
$$
where $ m_{1},\ldots,m_{k} $ are nonnegative integers.

If it holds that $ m_{1}+\cdots+m_{k}>p $, then $ X={\cal E} $.
\end{proposition}

\begin{proof}
It follows from Proposition~\ref{P-XXG} that
$$
X=X_{1}^{m_{1}}\otimes\cdots\otimes X_{k}^{m_{k}}
\le
\|X_{1}\|^{m_{1}}\otimes\cdots\otimes\|X_{k}\|^{m_{k}}
\otimes G^{m_{1}+\cdots+m_{k}},
$$
where $ G $ is the adjacency ($\varepsilon$--0)-matrix of the common
associated graph.

Since the graph is acyclic, it holds that $ G^{q}={\cal E} $ for all
$ q>p $. Therefore, with $ q=m_{1}+\cdots+m_{k}>p $, we arrive at the
inequality $ X \le {\cal E} $ which leads us to the desired result.
\qed
\end{proof}

\begin{lemma}\label{L-IXX}
Suppose that matrices $ X_{1},\ldots,X_{k} $ have a common associated
acyclic graph, and $ p $ is the length of the longest path in the graph.

If $ \|X_{i}\|\ge 0 $ for all $ i=1,\ldots,k $, then it holds
$$
\left\|\bigotimes_{i=1}^{k}(I\oplus X_{i})^{m_{i}}\right\|
\leq
\left(\bigoplus_{i=1}^{k}\|X_{i}\|\right)^{p}
$$
for any nonnegative integers $ m_{1},\ldots,m_{k} $.
\end{lemma}

\begin{proof}
Consider the matrix
\begin{eqnarray*}
X
&=&
\bigotimes_{i=1}^{k}(I\oplus X_{i})^{m_{i}}=
\bigoplus_{i_{1}=0}^{m_{1}}X_{1}^{i_{1}}\otimes\cdots\otimes
\bigoplus_{i_{k}=0}^{m_{k}}X_{k}^{i_{k}} \\
&=&
\bigoplus_{i_{1}=0}^{m_{1}}\cdots\bigoplus_{i_{k}=0}^{m_{k}}
X_{1}^{i_{1}}\otimes\cdots\otimes X_{k}^{i_{k}}
\le
\bigoplus_{0\leq i_{1}+\cdots+i_{k}\leq m}
X_{1}^{i_{1}}\otimes\cdots\otimes X_{k}^{i_{k}},
\end{eqnarray*}
where $ m=m_{1}+\cdots+m_{k} $. From Proposition~\ref{P-XEE} we may replace
$ m $ with $ p $ in the last term to get
$$
X \le  \bigoplus_{0\leq i_{1}+\cdots+i_{k}\leq p}
        X_{1}^{i_{1}}\otimes\cdots\otimes X_{k}^{i_{k}}.
$$

Proceeding to the norm, with its additive and multiplicative properties, we
arrive at the inequality
$$
\|X\|
\leq
\bigoplus_{0\leq i_{1}+\cdots+i_{k}\leq p}
\|X_{1}\|^{i_{1}}\otimes\cdots\otimes
\|X_{k}\|^{i_{k}}.
$$
Since for all $ i=1,\ldots,k $, it holds
$ 0\leq\|X_{i}\|\leq\|X_{1}\|\oplus\cdots\oplus\|X_{k}\| $, we finally have
$$
\|X\|
\leq
\bigoplus_{i=0}^{p}\left(\|X_{1}\|\oplus\cdots\oplus\|X_{k}\|\right)^{p}
=
\left(\bigoplus_{i=0}^{k}\|X_{i}\|\right)^{p}.
$$
\qed
\end{proof}

\section{An Algebraic Model of Queueing Networks}\label{S-AMQ}
We consider a network with $ n $ single-server nodes and customers of a
single class. The topology of the network is described by an oriented acyclic
graph $ {\cal G}=({\bf  N}, {\bf  A}) $, where the set
$ {\bf  N}=\{1,\ldots,n\} $ represents the nodes, and
$ {\bf  A}=\{(i,j)\} \subset {\bf  N} \times {\bf  N} $ does the arcs
determining the transition routes of customers.

For every node $ i \in {\bf  N} $, we denote the sets of its immediate
predecessors and successors respectively as
$ {\bf  P}(i)=\{ j | \, (j,i) \in {\bf  A} \} $ and
$ {\bf  S}(i)=\{ j | \, (i,j) \in {\bf  A} \} $. In specific cases, there
may be one of the conditions $ {\bf  P}(i)=\emptyset $ and
$ {\bf  S}(i)=\emptyset $ encountered. Each node $ i $ with
$ {\bf  P}(i)=\emptyset $ is assumed to represent an infinite external
arrival stream of customers; provided that $ {\bf  S}(i)=\emptyset $, it is
considered as an output node intended to release customers from the network.

Each node $ i \in {\bf  N} $ includes a server and its buffer with
infinite capacity, which together present a single-server queue operating
under the
first-come, first-served (FCFS) discipline. At the initial time, the server at
each node $ i $ is assumed to be free of customers, whereas in its buffer,
there may be $ r_{i} $, $ 0 \leq r_{i} \leq \infty $, customers waiting
for service. The value $ r_{i}=\infty $ is set for every node $ i $
with $ {\bf  P}(i)=\emptyset $, which represents an external arrival stream
of customers.

For the queue at node $ i $, we denote the $k$th arrival and departure
epochs respectively as $ u_{i}(k) $ and $ x_{i}(k) $. Furthermore, the
service time of the $k$th customer at server $ i $ is indicated by
$ \tau_{ik} $. We assume that $ \tau_{ik}\geq0 $ are given parameters
for all $ i=1,\ldots,n $, and $ k=1,2,\ldots $, while $ u_{i}(k) $ and
$ x_{i}(k) $ are considered as unknown state variables. With the condition
that the network starts operating at time zero, it is convenient to set
$ x_{i}(0) \equiv 0 $, and $ x_{i}(k) \equiv \varepsilon $ for all
$ k < 0 $, $ i=1,\ldots,n $.

It is easy to set up an equation which relates the system state variables. In
fact, the dynamics of any single-server node $ i $ with an infinite
buffer, operating on the FCFS basis, is described as
\begin{equation}\label{e-3}
x_{i}(k)=\tau_{ik} \otimes u_{i}(k) \oplus \tau_{ik} \otimes x_{i}(k-1).
\end{equation}
With the vector-matrix notations
$$
\mbox{\boldmath $u$}(k)
=\left(
   \begin{array}{c}
      u_{1}(k) \\
      \vdots   \\
      u_{n}(k)
   \end{array}
 \right),
\quad
\mbox{\boldmath $x$}(k)
=\left(
   \begin{array}{c}
      x_{1}(k) \\
      \vdots   \\
      x_{n}(k)
   \end{array}
 \right),
\quad
{\cal T}_{k}
=\left(
  \begin{array}{ccc}
    \tau_{1k}   &        & \varepsilon \\
                & \ddots &             \\
    \varepsilon &        & \tau_{nk}
  \end{array}
 \right),
$$
we may rewrite equation (\ref{e-3}) in a vector form, as
\begin{equation}\label{e-4}
\mbox{\boldmath $x$}(k)
={\cal T}_{k}\otimes\mbox{\boldmath $u$}(k)\oplus{\cal T}_{k}
\otimes\mbox{\boldmath $x$}(k-1).
\end{equation}

\subsection{Fork-Join Queueing Networks}

In fork-join networks, in addition to the usual service procedure, special
join and fork operations are performed in its nodes, respectively before and
after service. The join operation is actually thought to cause each customer
which comes into node $ i $, not to enter the buffer at the server but to
wait until at least one customer from every node $ j \in {\bf  P}(i) $
arrives. As soon as these customers arrive, they, taken one from each
preceding node, are united into one customer which then enters the buffer to
become a new member of the queue.

The fork operation at node $ i $ is initiated every time the service of a
customer is completed; it consists in giving rise to several new customers
instead of the original one. As many new customers appear in node $ i $ as
there are succeeding nodes included in the set $ {\bf  S}(i) $. These
customers simultaneously depart the node, each being passed to separate node
$ j \in {\bf  S}(i) $. We assume that the execution of fork-join operations
when appropriate customers are available, as well as the transition of
customers within and between nodes require no time.

As it immediately follows from the above description of the fork-join
operations, the $k$th arrival epoch into the queue at node $ i $ is
represented as
\begin{equation}\label{e-5}
u_{i}(k)=\left\{
           \begin{array}{ll}
	    {\displaystyle\bigoplus_{j\in {\bf  P}(i)}} x_{j}(k-r_{i}),
                          & \mbox{if $ {\bf  P}(i) \neq \emptyset $}, \\
            \varepsilon,  & \mbox{if $ {\bf  P}(i)=\emptyset $}.
	   \end{array}
          \right.
\end{equation}

In order to get this equation in a vector form, we first define the number
$ M=\max\{r_{i}|\, r_{i}<\infty, \, i=1,\ldots,n\} $. Now we may rewrite
(\ref{e-5}) as
$$
u_{i}(k)=\bigoplus_{m=0}^{M}\bigoplus_{j=1}^{n}g_{ji}^{m}\otimes x_{j}(k-m),
$$
where the numbers $ g_{ij}^{m} $ are determined by the condition
\begin{equation}\label{e-6}
g_{ij}^{m}=\left\{\begin{array}{ll}
               0, & \mbox{if $ i \in {\bf  P}(j) $ and $ m=r_{j} $}, \\
               \varepsilon, & \mbox{otherwise}.
             \end{array}\right.
\end{equation}

Let us introduce the matrices $ G_{m}=\left(g_{ij}^{m}\right) $ for each
$ m=0,1,\ldots,M $. In fact, $ G_{m} $ presents an adjacency matrix of
the partial graph $ {\cal G}_{m}=({\bf  N},{\bf  A}_{m}) $ with
$ {\bf  A}_{m}=\{(i,j)|\; (i,j)\in{\bf  A}; \; r_{j}=m\} $. Since the
graph of the entire network is acyclic, all its partial graphs
$ {\cal G}_{m} $, $ m=0,1,\ldots,M $, possess the same property.

With these matrices, equation (\ref{e-5}) may be written in the vector form
\begin{equation}\label{e-7}
\mbox{\boldmath $u$} (k)
=\bigoplus_{m=0}^{M} G_{m}^{T}\otimes\mbox{\boldmath $x$}(k-m),
\end{equation}
where $ G_{m}^{T} $ denotes the transpose of the matrix $ G_{m} $.

By combining equations (\ref{e-4}) and (\ref{e-7}), we arrive at the equation
\begin{eqnarray}
\lefteqn{\mbox{\boldmath $x$}(k)
={\cal T}_{k}\otimes G_{0}^{T}\otimes\mbox{\boldmath $x$}(k)
 \oplus{\cal T}_{k}\otimes\mbox{\boldmath $x$}(k-1)} \hspace{20mm}\nonumber \\
&& \oplus{\cal T}_{k}\otimes\bigoplus_{m=1}^{M} G_{m}^{T}
\otimes\mbox{\boldmath $x$}(k-m).\label{e-8}
\end{eqnarray}
Clearly, it is actually an implicit equation in $ \mbox{\boldmath $x$}(k) $,
which has the form of (\ref{e-1}), with
$ U={\cal T}_{k} \otimes G_{0}^{T} $. Taking into account that the matrix
$ {\cal T}_{k} $ is diagonal, one can prove the following statement
(see also \cite{Krivulin1996Themax-plus,Krivulin1996Max-plus}).

\begin{theorem}\label{T-XAX}
Suppose that in the fork-join network model, the graph $ {\cal G}_{0} $
associated with the matrix $ G_{0} $ is acyclic. Then equation
(\ref{e-8}) can be solved to produce the explicit dynamic state equation
\begin{equation}\label{e-9}
\mbox{\boldmath $x$}(k)
=\bigoplus_{m=1}^{M} A_{m}(k)\otimes\mbox{\boldmath $x$}(k-m),
\end{equation}
with the state transition matrices
\begin{eqnarray}
A_{1}(k) &=& (I\oplus{\cal T}_{k}\otimes G_{0}^{T})^{p}\otimes{\cal T}_{k}
                                  \otimes(I \oplus G_{1}^{T}),\label{e-10} \\
A_{m}(k) &=& (I\oplus{\cal T}_{k}\otimes G_{0}^{T})^{p}
       \otimes{\cal T}_{k}\otimes G_{m}^{T}, \quad m=2,\ldots,M,\label{e-11}
\end{eqnarray}
where $ p $ is the length of the longest path in $ {\cal G}_{0} $.
\end{theorem}


\subsection{Examples of Network Models}

An example of an acyclic fork-join network with $ n=5 $ is shown in
Fig.~\ref{F-AFJ}.
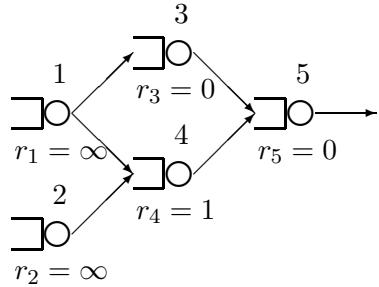
\begin{figure}[hhh]
\begin{center}
\setlength{\unitlength}{1mm}
\begin{picture}(48,36)
\newsavebox\queue
\savebox{\queue}(10,6){\thicklines
 \put(0,2){\line(1,0){4}}
 \put(0,-2){\line(1,0){4}}
 \put(4,2){\line(0,-1){4}}
 \put(6,0){\circle{3}}}

\put(0,19){\usebox\queue}
\put(5,26){$1$}
\put(0,16){$r_{1}=\infty$}
\put(8,22){\vector(1,1){8}}
\put(8,22){\vector(1,-1){8}}
\put(0,3){\usebox\queue}
\put(5,10){$2$}
\put(0,0){$r_{2}=\infty$}
\put(8,6){\vector(1,1){8}}

\put(16,27){\usebox\queue}
\put(21,34){$3$}
\put(16,24){$r_{3}=0$}
\put(24,30){\vector(1,-1){8}}
\put(16,11){\usebox\queue}
\put(21,18){$4$}
\put(16,8){$r_{4}=1$}
\put(24,14){\vector(1,1){8}}

\put(32,19){\usebox\queue}
\put(37,26){$5$}
\put(32,16){$r_{5}=0$}
\put(40,22){\vector(1,0){8}}

\end{picture}
\end{center}
\caption{An acyclic fork-join network.}\label{F-AFJ}
\end{figure}

Since for the network $ M=1 $, we have from (\ref{e-6})
$$
G_{0}=\left(\begin{array}{ccccc}
        \varepsilon & \varepsilon & 0           & \varepsilon & \varepsilon \\
        \varepsilon & \varepsilon & \varepsilon & \varepsilon & \varepsilon \\
        \varepsilon & \varepsilon & \varepsilon & \varepsilon & 0 \\
        \varepsilon & \varepsilon & \varepsilon & \varepsilon & 0 \\
        \varepsilon & \varepsilon & \varepsilon & \varepsilon & \varepsilon
      \end{array}\right),
\qquad
G_{1}=\left(\begin{array}{ccccc}
        \varepsilon & \varepsilon & \varepsilon & 0           & \varepsilon \\
        \varepsilon & \varepsilon & \varepsilon & 0           & \varepsilon \\
        \varepsilon & \varepsilon & \varepsilon & \varepsilon & \varepsilon \\
        \varepsilon & \varepsilon & \varepsilon & \varepsilon & \varepsilon \\
        \varepsilon & \varepsilon & \varepsilon & \varepsilon & \varepsilon
      \end{array}\right).
$$

Taking into account that for the graph $ {\cal G}_{0} $, the length of its
longest path $ p=2 $, we arrive at the dynamic equation
$$
\mbox{\boldmath $x$}(k)=A(k)\otimes\mbox{\boldmath $x$}(k-1),
$$
with the state transition matrix calculated from (\ref{e-10}) as
\begin{eqnarray*}
\lefteqn{A(k)=(I\oplus{\cal T}_{k}\otimes G_{0}^{T})^{2}\otimes{\cal T}_{k}
                                       \otimes(I \oplus G_{1}^{T})} \\
& &
=  \left(\begin{array}{ccccc}
       \tau_{1k}   & \varepsilon & \varepsilon & \varepsilon & \varepsilon \\
       \varepsilon & \tau_{2k}   & \varepsilon & \varepsilon & \varepsilon \\
       \tau_{1k}\!\otimes\!\tau_{3k}
                   & \varepsilon & \tau_{3k}   & \varepsilon & \varepsilon \\
       \tau_{4k}   & \tau_{4k}   & \varepsilon & \tau_{4k}   & \varepsilon \\
       (\tau_{1k}\!\otimes\!\tau_{3k}\!\oplus\!\tau_{4k})\!\otimes\!\tau_{5k}
                   & \tau_{4k}\!\otimes\!\tau_{5k}
                                 & \tau_{3k}\!\otimes\!\tau_{5k}
                                               & \tau_{4k}\!\otimes\!\tau_{5k}
                                                             & \tau_{5k}
   \end{array}\right).
\end{eqnarray*}

Note that open tandem queueing systems (see Fig.~\ref{F-OTQ}) can be
considered as trivial networks in which no fork and join operations are
actually performed.

\begin{figure}[hhh]
\begin{center}
\setlength{\unitlength}{1mm}
\begin{picture}(64,12)
\savebox{\queue}(10,6){\thicklines
 \put(0,2){\line(1,0){4}}
 \put(0,-2){\line(1,0){4}}
 \put(4,2){\line(0,-1){4}}
 \put(6,0){\circle{3}}}

\put(0,3){\usebox\queue}
\put(5,10){$1$}
\put(0,0){$r_{1}=\infty$}
\put(8,6){\vector(1,0){8}}
\put(16,3){\usebox\queue}
\put(21,10){$2$}
\put(16,0){$r_{2}=0$}
\put(24,6){\vector(1,0){8}}
\multiput(34,6)(2,0){3}{\circle*{1}}
\put(40,6){\vector(1,0){8}}
\put(48,3){\usebox\queue}
\put(53,10){$n$}
\put(48,0){$r_{n}=0$}
\put(56,6){\vector(1,0){8}}

\end{picture}
\end{center}
\caption{Open tandem queues.}\label{F-OTQ}
\end{figure}
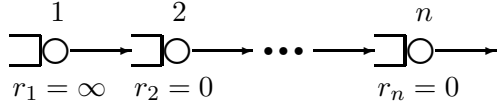

For the system in Fig.~\ref{F-OTQ}, we have $ M=0 $, and $ p=n-1 $. Its
related state transition matrix $ A(k) $ has the entries \cite{Krivulin1994Using,Krivulin1995Amax-algebra}
$$
a_{ij}(k)
=\left\{
  \begin{array}{ll}
   \tau_{jk}\otimes\tau_{j+1k}\otimes\cdots\otimes\tau_{ik},
                        & \mbox{if $ i\geq j$}, \\
   \varepsilon, & \mbox{otherwise}.
  \end{array}
 \right.
$$

\section{A Monotonicity Property}\label{S-MPR}
In this section, a property of monotonicity is established which shows how the
system state vector $ \mbox{\boldmath $x$}(k) $ may vary with the initial
numbers of customers $ r_{i} $. It is actually proven that the entries of
$ \mbox{\boldmath $x$}(k) $ for all $ k=1,2,\ldots $, do not decrease
when the numbers $ r_{i} $ with $ 0<r_{i}<\infty $, $ i=1,\ldots,n $,
are reduced to zero.

As it is easy to see, the change in the initial numbers of customers
results only in modifications to partial graphs $ {\cal G}_{m} $ and so to
their adjacency matrices $ G_{m} $. Specifically, reducing these numbers to
zero leads us to new matrices
$ \widetilde{G}_{0}=G_{0}\oplus G_{1}\cdots\oplus G_{M} $, and
$ \widetilde{G}_{m}={\cal E} $ for all $ m=1,\ldots,M $.

We start with a lemma which shows that replacing the numbers
$ r_{i}=1 $ with $ r_{i}=0 $ does not decrease the entries of the
matrix $ A_{1}(k) $ defined by (\ref{e-10}).

\begin{lemma}\label{L-ALA}
For all $ k=1,2,\ldots $, it holds
$$
A_{1}(k)\leq\widetilde{A}(k)
$$
with $ \widetilde{A}(k)
=(I\oplus{\cal T}_{k}\otimes\widetilde{G}_{0}^{T})^{q}\otimes{\cal T}_{k} $,
where $ \widetilde{G}_{0}=G_{0}\oplus G_{1} $, and $ q $ is the length
of the longest path in the graph associated with the matrix
$ \widetilde{G}_{0} $.
\end{lemma}
\begin{proof}
Consider the matrix $ A_{1}(k) $ and represent it in the form
$$
A_{1}(k)
=
((I\oplus{\cal T}_{k}\otimes G_{0}^{T})^{p}\otimes{\cal T}_{k})
\oplus((I\oplus{\cal T}_{k}\otimes G_{0}^{T})^{p}\otimes{\cal T}_{k}
\otimes G_{1}^{T}),
$$
where $ p $ is the length of the longest path in the graph associated
with $ G_{0} $.

As one can see, to prove the lemma, it will suffice to verify both
inequalities
\begin{eqnarray}
\widetilde{A}(k)
& \geq &
(I\oplus{\cal T}_{k}\otimes G_{0}^{T})^{p}\otimes{\cal T}_{k},\label{e-12} \\
\widetilde{A}(k)
& \geq &
(I\oplus{\cal T}_{k}\otimes G_{0}^{T})^{p}\otimes{\cal T}_{k}
\otimes G_{1}^{T}.\label{e-13}
\end{eqnarray}

Let us write the obvious representation
$$
(I\oplus{\cal T}_{k}\otimes\widetilde{G}_{0}^{T})^{q}
=\bigoplus_{i=0}^{q}
(I\oplus{\cal T}_{k}\otimes G_{0}^{T})^{i}
\otimes({\cal T}_{k}\otimes G_{1}^{T})^{q-i}.
$$
Since $ q\geq p $, we get from the representation
$$
(I\oplus{\cal T}_{k}\otimes\widetilde{G}_{0}^{T})^{q}
\geq (I\oplus{\cal T}_{k}\otimes\widetilde{G}_{0}^{T})^{p}
=((I\oplus{\cal T}_{k}\otimes G_{0}^{T})
\oplus{\cal T}_{k}\otimes G_{1}^{T})^{p}
\geq (I\oplus{\cal T}_{k}\otimes G_{0}^{T})^{p}.
$$
It remains to multiply both sides of the above inequality by
$ {\cal T}_{k} $ on the right so as to arrive at (\ref{e-12}).

To verify (\ref{e-13}), let us first assume that $ q>p $. In this case,
we obtain
\begin{eqnarray*}
\lefteqn{(I\oplus{\cal T}_{k}\otimes\widetilde{G}_{0}^{T})^{q}
\geq
(I\oplus{\cal T}_{k}\otimes\widetilde{G}_{0}^{T})^{p+1}} \\
& & =
((I\oplus{\cal T}_{k}\otimes G_{0}^{T})
\oplus{\cal T}_{k}\otimes G_{1}^{T})^{p+1}
\geq
(I\oplus{\cal T}_{k}\otimes G_{0}^{T})^{p}\otimes{\cal T}_{k}
\otimes G_{1}^{T}.
\end{eqnarray*}

Suppose now that $ q=p $. Then it is necessary that
$ G_{1}\otimes G_{0}^{p}={\cal E} $. If this were not the case, there would
be a path in the graph associated with the matrix
$ \widetilde{G}_{0}=G_{0}\oplus G_{1} $, which has its length greater than
$ p $, and we would have $ q>p $.

Clearly, the condition $ G_{1}\otimes G_{0}^{p}={\cal E} $ results in
$ ({\cal T}_{k}\otimes G_{0}^{T})^{p}\otimes{\cal T}_{k}\otimes G_{1}^{T}
={\cal E} $, and thus we get
\begin{eqnarray*}
\lefteqn{(I\oplus{\cal T}_{k}\otimes\widetilde{G}_{0}^{T})^{q}
=((I\oplus{\cal T}_{k}\otimes G_{0}^{T})
\oplus{\cal T}_{k}\otimes G_{1}^{T})^{p}} \\
& & \geq
(I\oplus{\cal T}_{k}\otimes G_{0}^{T})^{p-1}\otimes{\cal T}_{k}
\otimes G_{1}^{T}
=
(I\oplus{\cal T}_{k}\otimes G_{0}^{T})^{p}\otimes{\cal T}_{k}
\otimes G_{1}^{T}.
\end{eqnarray*}
Since it holds
$ (I\oplus{\cal T}_{k}\otimes\widetilde{G}_{0}^{T})^{p}\otimes{\cal T}_{k}
\geq (I\oplus{\cal T}_{k}\otimes\widetilde{G}_{0}^{T})^{p} $, one can conclude
that inequality (\ref{e-13}) is also valid.
\qed
\end{proof}

\begin{theorem}\label{T-MON}
In the acyclic fork-join queueing network model (\ref{e-9}--\ref{e-11}),
reducing the initial numbers of customers from any finite values to zero does
not decrease the entries of the system state vector
$ \mbox{\boldmath $x$}(k) $ for all $ k=1,2,\ldots $.
\end{theorem}

\begin{proof}
Let $ \mbox{\boldmath $x$}(k) $ be determined by (\ref{e-9}--\ref{e-11}).
Suppose that the vector $ \widetilde{\mbox{\boldmath $x$}}(k) $ satisfies
the dynamic equation
$$
\widetilde{\mbox{\boldmath $x$}}(k)
=\widetilde{A}(k)\otimes\widetilde{\mbox{\boldmath $x$}}(k-1)
$$
with
$$
\widetilde{A}(k)
=\left(I\oplus{\cal T}_{k}\otimes\bigoplus_{m=0}^{M}G_{m}^{T}\right)^{q}
\otimes{\cal T}_{k}
=(I\oplus{\cal T}_{k}\otimes G^{T})^{q}\otimes{\cal T}_{k},
$$
where $ q $ is the length of the longest path in the graph associated with
the matrix $ G=G_{0}\oplus G_{1}\oplus\cdots\oplus G_{m} $.

Now we have to show that for all $ k=1,2,\ldots $, it holds
$$
\mbox{\boldmath $x$}(k)\leq\widetilde{\mbox{\boldmath $x$}}(k).
$$

Since $ \mbox{\boldmath $x$}(k_{1}) \leq \mbox{\boldmath $x$}(k_{2}) $
for any $ k_{1}<k_{2} $, we have from (\ref{e-9})
$$
\mbox{\boldmath $x$}(k)
=\bigoplus_{m=1}^{M} A_{m}(k)\otimes\mbox{\boldmath $x$}(k-m)
\leq\left(\bigoplus_{m=1}^{M}A_{m}(k)\right)\otimes\mbox{\boldmath $x$}(k-1).
$$

Consider the matrix
$$
\widetilde{A}_{1}(k)=\bigoplus_{m=1}^{M}A_{m}(k)
=(I\oplus{\cal T}_{k}\otimes G_{0}^{T})^{p}\otimes{\cal T}_{k}
  \otimes\left(I\oplus \bigoplus_{m=1}^{M}G_{m}^{T}\right).
$$
By applying Lemma~\ref{L-ALA}, we have
$$
\widetilde{A}_{1}(k)
\leq\left(I\oplus{\cal T}_{k}\otimes G_{0}^{T}
\oplus\bigoplus_{m=1}^{M}G_{m}^{T}\right)^{q}\otimes{\cal T}_{k}
=\widetilde{A}(k).
$$

Starting with the condition
$ \mbox{\boldmath $x$}( 0)=\widetilde{\mbox{\boldmath $x$}}(0) $, we
successively verify that the relations
$$
\mbox{\boldmath $x$}(k)
\leq
\widetilde{A}_{1}(k)\otimes\mbox{\boldmath $x$}(k -1)
\leq
\widetilde{A}(k)\otimes\mbox{\boldmath $x$}(k-1)
\leq
\widetilde{A}(k)\otimes\widetilde{\mbox{\boldmath $x$}}(k-1)
=
\widetilde{\mbox{\boldmath $x$}}(k)
$$
are valid for each $ k=1,2,\ldots $.
\qed
\end{proof}

\section{Bounds on the Service Cycle Completion Time}\label{S-BSC}
We consider the evolution of the system as a sequence of service cycles:
the 1st cycle starts at the initial time, and it is terminated as soon as all
the servers in the network complete their 1st service, the 2nd cycle is
terminated as soon as the servers complete their 2nd service, and so on.
Clearly, the completion time of the $k$th cycle can be represented as
$$
\max_{i}x_{i}(k)=\|\mbox{\boldmath $x$}(k)\|.
$$

The next lemma provides simple lower and upper bounds for the $k$th cycle
completion time.
\begin{lemma}\label{L-TXT}
For all $ k=1,2,\ldots $, it holds
$$
\left\|\sum_{i=1}^{k}{\cal T}_{i}\right\|
\le
\|\mbox{\boldmath $x$}(k) \|
\le
\sum_{i=1}^{k}\|{\cal T}_{i}\|
+ p\left(\bigoplus_{i=1}^{k}\|{\cal T}_{i}\|\right).
$$
\end{lemma}
\begin{proof}
To prove the left inequality first note that
$$
A_{1}(k)=(I\oplus{\cal T}_{k}\otimes G_{0}^{T})^{p}\otimes{\cal T}_{k}
\otimes(I \oplus G_{1}^{T}) \geq {\cal T}_{k}.
$$
With this condition, we have from (\ref{e-9})
$$
\mbox{\boldmath $x$}(k)=\bigoplus_{m=1}^{M}A_{m}(k)
\otimes\mbox{\boldmath $x$}(k-m)
\geq A_{1}(k)\otimes\mbox{\boldmath $x$}(k-1)
\geq {\cal T}_{k}\otimes\mbox{\boldmath $x$}(k-1).
$$

Now we can write
$$
\mbox{\boldmath $x$}(k)
\geq {\cal T}_{k}\otimes\mbox{\boldmath $x$}(k-1)
\geq {\cal T}_{k}\otimes{\cal T}_{k-1}\otimes\mbox{\boldmath $x$}(k-2)
\geq\cdots
\geq {\cal T}_{k}\otimes\cdots\otimes{\cal T}_{1}
\otimes\mbox{\boldmath $x$}(0),
$$
where $ \mbox{\boldmath $x$}(0)={\bf 0} $. Taking the norm, and considering
that $ {\cal T}_{i} $, $ i=1,\ldots,k $, present diagonal matrices, we get
$$
\|\mbox{\boldmath $x$}(k)\|
\ge
\|{\cal T}_{k}\otimes\cdots\otimes{\cal T}_{1}\|
=
\|{\cal T}_{1}+\cdots+{\cal T}_{k}\|.
$$

To obtain an upper bound, let us replace the general system
(\ref{e-9}--\ref{e-11}) with that governed by the equation
\begin{equation}\label{e-14}
\widetilde{\mbox{\boldmath $x$}}(k)
=\widetilde{A}(k)\otimes\widetilde{\mbox{\boldmath $x$}}(k-1)
\end{equation}
with $ \widetilde{A}(k)=(I\oplus{\cal T}_{k}
\otimes\widetilde{G}^{T})^{q}\otimes{\cal T}_{k} $, where
$ \widetilde{G}=G_{0}\oplus G_{1}\oplus\cdots\oplus G_{m} $, and $ q $
is the length of the longest path in the graph associated with
$ \widetilde{G} $. As it follows from Theorem~\ref{T-MON}, one has
$ \mbox{\boldmath $x$}(k) \leq\widetilde{\mbox{\boldmath $x$}}(k) $ for
all $ k=1,2,\ldots $.

Let us denote
$ \widetilde{A}_{k}=\widetilde{A}(k)\otimes\cdots\otimes\widetilde{A}(1) $.
With the condition
$ \widetilde{\mbox{\boldmath $x$}}(0)=\mbox{\boldmath $x$}( 0)={\bf 0} $, we
get from (\ref{e-14})
$$
\|\widetilde{\mbox{\boldmath $x$}}(k)\|
=\|\widetilde{A}(k)\otimes\cdots\otimes\widetilde{A}(1)\|
=\|\widetilde{A}_{k}\|.
$$

With Proposition~\ref{P-XEE} we have
$$
\widetilde{A}_{k}
=
\bigotimes_{i=1}^{k}(I\oplus{\cal T}_{k-i+1}
\otimes\widetilde{G}^{T})^{q}\otimes{\cal T}_{k-i+1}
\le
\bigotimes_{i=1}^{k}\|{\cal T}_{i}\|
\otimes
\bigotimes_{i=1}^{k}(I\oplus{\cal T}_{k-i+1}
\otimes\widetilde{G}^{T})^{q}.
$$

Proceeding to the norm and using Lemma~\ref{L-IXX}, we arrive at the
inequality
$$
\|\widetilde{A}_{k}\|
\le
\bigotimes_{i=1}^{k}\|{\cal T}_{i}\|\otimes
\left(\bigoplus_{i=1}^{k}\|{\cal T}_{i}\|\right)^{q}
=
\sum_{i=1}^{k}\|{\cal T}_{i}\|
+q\left(\bigoplus_{i=1}^{k}\|{\cal T}_{i}\|\right).
$$
which provides us with the desired result.
\qed
\end{proof}

\section{Stochastic Extension of the Network Model}\label{S-SEN}

Suppose that for each node $ i=1,\ldots,n $, the service times
$ \tau_{i1},\tau_{i2},\ldots $, form a sequence of independent and
identically distributed (i.i.d.) non-negative random variables with
$ \mathbb{E}[\tau_{ik}]<\infty $ and
$ \mathbb{D}[\tau_{ik}]<\infty $ for all $ k=1,2,\ldots $.

As a performance measure of the stochastic network model, we consider
the mean cycle time which is defined as
\begin{equation}\label{e-G}
\gamma=\lim_{k\to\infty}\frac{1}{k}\|\mbox{\boldmath $x$}(k)\|
\end{equation}
provided that the above limit exists. Another performance measure of interest
is the throughput defined as $ \pi=1/\gamma $.

Since it is frequently rather difficult to evaluate the mean cycle time
exactly, even though the network under study is quite simple, one can try to
derive bounds on $ \gamma $. In this section, we show how these bounds may
be obtained based on (max,+)-algebra representation of the network dynamics.

We start with some preliminary results which include properties of the
expectation operator, formulated in terms of (max,+)-algebra operations.

\subsection{Some Properties of Expectation}

Let $ \xi_{1},\ldots,\xi_{k} $ be random variables taking their values in
$ \underline{\mathbb{R}} $, and such that their expected values
$ \mathbb{E}[\xi_{i}] $, $ i=1,\ldots,k $, exist.

First note that ordinary properties of expectation leads us to the obvious
relations
$$
\mathbb{E}\left[\bigoplus_{i=1}^{k}\xi_{i}\right]
\leq\bigotimes_{i=1}^{k}\mathbb{E}[\xi_{i}],
\qquad
\mbox{and}\qquad
\mathbb{E}\left[\bigotimes_{i=1}^{k}\xi_{i}\right]
=\bigotimes_{i=1}^{k}\mathbb{E}[\xi_{i}].
$$

Furthermore, the next statement is valid.
\begin{lemma}\label{L-EEE}
It holds
$$
\mathbb{E}\left[\bigoplus_{i=1}^{k}\xi_{i}\right]
\geq \bigoplus_{i=1}^{k}\mathbb{E}[\xi_{i}].
$$
\end{lemma}
\begin{proof}
The statement of the lemma for $ k=2 $ follows immediately from the
identity
$$
x\oplus y=\frac{1}{2}(x+y+|x-y|), \quad \mbox{for all
$ x,y\in\mathbb{R}$}
$$
and ordinary properties of expectation. It remains to extend the statement to
the case of arbitrary $ k $ by induction.
\qed
\end{proof}

The next result \cite{Gumbel1954Themaxima,Hartley1954Universal} provides an upper
bound for the expected value of the maximum of i.i.d. random variables.
\begin{lemma}\label{L-GDH}
Let $ \xi_{1},\ldots,\xi_{k} $ be i.i.d. random variables with
$ \mathbb{E}[\xi_{1}]<\infty $ and
$ \mathbb{D}[\xi_{1}]<\infty $. Then it holds
$$
\mathbb{E}\left[\bigoplus_{i=1}^{k}\xi_{i}\right]
\le
\mathbb{E}[\xi_{1}]+\frac{k-1}{\sqrt{2k-1}}\sqrt{\mathbb{D}[\xi_{1}]}.
$$
\end{lemma}

Consider a random matrix $ X $ with its entries $ x_{ij} $ taking
values in $ \underline{\mathbb{R}} $. We denote by $ \mathbb{E}[X] $
the matrix obtained from $ X $ by replacing each entry $ x_{ij} $ by
its expected value $ \mathbb{E}[x_{ij}] $.

\begin{lemma}\label{L-EXE}
It holds
$$
\mathbb{E}\|X\|\geq\|\mathbb{E}[X]\|.
$$
\end{lemma}

\begin{proof}
It follows from Lemma~\ref{L-EEE} that
$$
\mathbb{E}\|X\|=\mathbb{E}\left[\bigoplus_{i,j}x_{ij}\right]
\geq \bigoplus_{i,j}\mathbb{E}[x_{ij}]=\|\mathbb{E}[X]\|.
$$
\qed
\end{proof}

\subsection{Existence of the Mean Cycle Time}

In the analysis of the mean cycle time, one first has to convince himself that
the limit at (\ref{e-G}) exists. As a standard tool to verify the existence of
the above limit, the next theorem proposed in \cite{Kingman1973Subadditive} is normally
applied. One can find examples of the implementation of the theorem in the
$(\max,+)$-algebra framework in \cite{Baccelli1991Estimates,Baccelli1993Synchronization}.

\begin{theorem}\label{T-KIN}
Let $ \{\xi_{lk} | \; l,k=0,1,\ldots; \, l<k \} $ be a family of random
variables which satisfy the following properties:

Subadditivity: $ \xi_{lk} \leq \xi_{lm} + \xi_{mk} $ for all
$ l<m<k $;

Stationarity: both families $ \{\xi_{l+1 k+1} | \; l<k \} $ and
$ \{\xi_{lk} | \; l<k \} $ have the same joint distributions;

Boundedness: for all $ k=1,2,\ldots $, there exists
$ \mathbb{E}[\xi_{0k}] \geq -ck $ for some finite number $ c $.

Then there exists a constant $ \gamma $, such that it holds
\begin{enumerate}
\item $ \displaystyle{\lim_{k\rightarrow\infty} \xi_{0k}/k=\gamma} $
\quad with probability 1,
\item $ \displaystyle{\lim_{k\rightarrow\infty}
\mathbb{E}[\xi_{0k}]/k=\gamma}$.
\end{enumerate}
\end{theorem}

For simplicity, we examine the existence of the mean cycle time for a network
with the maximum of the initial numbers of customers in nodes $ M\leq1 $. As
it follows from representation (\ref{e-9}--\ref{e-11}), the dynamics of the
system may be described by the equation
$$
\mbox{\boldmath $x$}(k)=A(k)\otimes\mbox{\boldmath $x$}(k-1)
$$
with the matrix $ A(k)=A_{1}(k) $ determined by (\ref{e-10}).
Clearly, in the case of $ M>1 $, a similar representation can be easily
obtained by going to an extended model with a new state vector which combines
several consecutive state vectors of the original system.

To prove the existence of the mean cycle time, first note that
$ \tau_{ik} $ with $ k=1,2,\ldots $, are i.i.d. random variables for
each $ i=1,\ldots,n $, and consequently, $ {\cal T}_{k} $ are i.i.d.
random matrices, whereas $ \|{\cal T}_{k}\| $ present i.i.d. random
variables with $ \mathbb{E}\|{\cal T}_{k}\|<\infty $ and
$ \mathbb{D}\|{\cal T}_{k}\|<\infty $ for all $ k=1,2,\ldots $.

Furthermore, since the matrix $ A(k) $ depends only on $ {\cal T}_{k} $,
the matrices $ A(1),A(2),\ldots $, also present i.i.d. random matrices. It
is easy to verify that $ 0\leq\mathbb{E}\|A(k)\|<\infty $ for all
$ k=1,2,\ldots $.

In order to apply Theorem~\ref{T-KIN} to stochastic system (\ref{e-9}) with
transition matrix (\ref{e-10}), one can define the family of random variables
$ \{\xi_{lk} | \; l<k \} $ with
$$
\xi_{lk}=\|A(k)\otimes\cdots\otimes A(l+1)\|.
$$

Since $ A(i) $, $ i=1,2,\ldots $, present i.i.d. random matrices, the
family $ \{\xi_{lk}|\;l<k\} $ satisfies the stationarity condition of
Theorem~\ref{T-KIN}. Furthermore, the multiplicative property of the norm
endows the family with subadditivity. The boundedness condition can be readily
verified based on the condition that
$ 0\leq\mathbb{E}[\tau_{ik}]<\infty $ for all $ i=1,\ldots,n $, and
$ k=1,2,\ldots $.

\subsection{Calculating Bounds on the Mean Cycle Time}
Now we are in a position to present our main result which offers bounds on the
mean cycle time.
\begin{theorem}\label{T-TGT}
In the stochastic dynamical system (\ref{e-9}) the mean cycle time
$ \gamma $ satisfies the double inequality
\begin{equation}\label{e-15}
\|\mathbb{E}[{\cal T}_{1}]\| \leq \gamma \leq
\mathbb{E}\|{\cal T}_{1}\|.
\end{equation}
\end{theorem}

\begin{proof}
Since Theorem~\ref{T-KIN} hods true, we may write
$$
\gamma=
\lim_{k\rightarrow\infty}\frac{1}{k}\mathbb{E}\|\mbox{\boldmath $x$}(k)\|.
$$

Let us first prove the left inequality in (\ref{e-15}). From
Lemmas~\ref{L-TXT} and \ref{L-EXE}, we have
$$
\frac{1}{k}\mathbb{E}\|\mbox{\boldmath $x$}(k)\|
\geq
\frac{1}{k}\mathbb{E}\left\|\sum_{i=1}^{k}{\cal T}_{i}\right\|
\geq
\left\|\frac{1}{k}\sum_{i=1}^{k}\mathbb{E}[{\cal T}_{i}]\right\|
=\|\mathbb{E}[{\cal T}_{1}]\|,
$$
independently of $ k $.

With the upper bound offered by Lemma~\ref{L-TXT}, we get
$$
\frac{1}{k}\mathbb{E}\|\mbox{\boldmath $x$}(k)\|
\leq
\mathbb{E}\|{\cal T}_{1}\|
+\frac{p}{k}\mathbb{E}\left[\bigoplus_{i=1}^{k}\|{\cal T}_{i}\|\right].
$$
From Lemma~\ref{L-GDH}, the second term on the right-hand side may be replaced
by that of the form
$$
\frac{p}{k}\left(\mathbb{E}\|{\cal T}_{1}\|
+\frac{k-1}{\sqrt{2k-1}}\sqrt{\mathbb{D}\|{\cal T}_{1}\|}\right),
$$
which tends to $0$ as $ k\rightarrow\infty $.
\qed
\end{proof}

\section{Discussion and Examples}\label{S-DEX}

Now we discuss the behaviour of the bounds (\ref{e-15}) under various
assumptions concerning the service times in the network. First note that the
derivation of the bounds does not require the $k$th service times
$ \tau_{ik} $ to be independent for all $ i=1,\ldots,n $. As it is easy
to see, if $ \tau_{ik}=\tau_{k} $ for all $ i $, we have
$ \|\mathbb{E}[{\cal T}]_{1}\|=\mathbb{E}\|{\cal T}_{1}\| $, and
so the lower and upper bound coincide.

To show how the bounds vary with strengthening the dependency, we consider the
network with $ n=5 $ nodes, depicted in Fig.~\ref{F-AFJ}. Let
$ \tau_{i1}=\sum_{j=1}^{5}a_{ij}\xi_{j1} $, where $ \xi_{j1} $,
$ j=1,\ldots,5 $, are i.i.d. random variables with the exponential
distribution of mean $ 1 $, and
$$
a_{ij}
=\left\{\begin{array}{ll}
    a,                & \mbox{if $ i=j$}, \\
    \frac{1}{4}(1-a), & \mbox{if $ i\ne j$},
 \end{array}\right.
$$
where $ a $ is a number such that $ 1\leq a\leq 1/5 $.

It is evident that for $ a=1 $, one has $ \tau_{i1}=\xi_{i1} $, and
then $ \tau_{i1} $, $ i=1,\ldots,5 $, present independent random
variables. As $ a $ decreases, the service times $ \tau_{i1} $ become
dependent, and with $ a=1/5 $, we will have
$ \tau_{i1}=(\xi_{11}+\cdots+\xi_{51})/5 $ for all $ i=1,\ldots,5 $.

Table~\ref{B-DEP} presents estimates of the mean cycle time
$ \widehat{\gamma} $ obtained via simulation after performing 100000
service cycles, together with the corresponding lower and upper bounds
calculated from (\ref{e-15}).
\begin{table}[hhh]
\begin{center}
\begin{tabular}{||c|c|c|c||}
\hline\hline
\strut & & & \\
$\; a \;$ &
$\|\mathbb{E}[{\cal T}_{1}]\|$ &
$\widehat{\gamma}$ &
$\mathbb{E}\|{\cal T}_{1}\|$ \\
\strut & & & \\
\hline
1   & 1.0 & 1.005718 & 2.283333 \\
1/2 & 1.0 & 1.002080 & 1.481250 \\
1/3 & 1.0 & 1.000871 & 1.213889 \\
1/4 & 1.0 & 1.000279 & 1.080208 \\
1/5 & 1.0 & 1.000000 & 1.000000 \\
\hline\hline
\end{tabular}
\caption{Numerical results for a network with dependent service times.}
\label{B-DEP}
\end{center}
\end{table}

Let us now consider the network in Fig.~\ref{F-AFJ} under the assumption that
the service times $ \tau_{i1} $ are independent exponentially distributed
random variables. We suppose that $ \mathbb{E}[\tau_{i1}]=1 $ for all
$ i $ except for one, say $ i=4 $, with $ \mathbb{E}[\tau_{41}] $
essentially greater than $ 1 $. One can see that the difference between the
upper and lower bounds will decrease as the value of
$ \mathbb{E}[\tau_{41}] $ increases. Table~\ref{B-DOM} shows how the
bounds vary with different values of $ \mathbb{E}[\tau_{41}] $.
\begin{table}[hhh]
\begin{center}
\begin{tabular}{||c|c|c|c||}
\hline\hline
\strut & & & \\
$\mathbb{E}[\tau_{41}]$ &
$\|\mathbb{E}[{\cal T}_{1}]\|$ &
$\widehat{\gamma}$ &
$\mathbb{E}\|{\cal T}_{1}\|$ \\
\strut & & & \\
\hline
 1.0 &  1.0 &  1.005718 &  2.283333 \\
 2.0 &  2.0 &  2.004857 &  2.896032 \\
 3.0 &  3.0 &  3.004242 &  3.685531 \\
 4.0 &  4.0 &  4.003627 &  4.554525 \\
 5.0 &  5.0 &  5.003013 &  5.465368 \\
 6.0 &  6.0 &  6.002398 &  6.400835 \\
 7.0 &  7.0 &  7.001783 &  7.351985 \\
 8.0 &  8.0 &  8.001168 &  8.313731 \\
 9.0 &  9.0 &  9.000553 &  9.282968 \\
10.0 & 10.0 & 10.000008 & 10.257692 \\
\hline\hline
\end{tabular}
\caption{Results for a network with a dominating service time.}\label{B-DOM}
\end{center}
\end{table}

Let us discuss the effect of decreasing the variance
$ \mathbb{D}[\tau_{i1}] $ on the bounds on $ \gamma $. Note that if
$ \tau_{i1} $ were degenerate random variables with zero variance, the
lower and upper bounds in (\ref{e-15}) would coincide. One can therefore
expect that with decreasing the variance of $ \tau_{i1} $, the accuracy of
the bounds increases.

As an illustration, consider a tandem queueing system (see Fig.~\ref{F-OTQ})
with $ n=5 $ nodes. Suppose that $ \tau_{i1}=\xi_{i1}/r $, where
$ \xi_{i1} $, $ i=1,\ldots,5 $, are i.i.d. random variables which have the
Erlang distribution with the probability density function
$$
f_{r}(t)
=\left\{
   \begin{array}{ll}
     t^{r-1}e^{-t}/(r-1)!, & \mbox{if $ t>0$}, \\
     0,                    & \mbox{if $ t\leq 0$}.
   \end{array}
 \right.
$$
Clearly, $ \mathbb{E}[\tau_{i1}]=1 $ and
$ \mathbb{D}[\tau_{i1}]=1/r $. Related numerical results including
estimates $ \widehat{\gamma} $ evaluated by simulating 100000 cycles are
shown in Table~\ref{B-CHA}.
\begin{table}[hhh]
\begin{center}
\begin{tabular}{||c|c|c|c||}
\hline\hline
\strut & & & \\
$\; r \;$ &
$\|\mathbb{E}[{\cal T}_{1}]\|$ &
$\widehat{\gamma}$ &
$\mathbb{E}\|{\cal T}_{1}\|$ \\
\strut & & & \\
\hline
 1 & 1.0 & 1.042476 & 2.928968 \\
 2 & 1.0 & 1.026260 & 2.311479 \\
 3 & 1.0 & 1.019503 & 2.045538 \\
 4 & 1.0 & 1.015637 & 1.890824 \\
 5 & 1.0 & 1.013110 & 1.787242 \\
 6 & 1.0 & 1.010864 & 1.711943 \\
 7 & 1.0 & 1.009920 & 1.654154 \\
 8 & 1.0 & 1.008409 & 1.608064 \\
 9 & 1.0 & 1.007726 & 1.570232 \\
10 & 1.0 & 1.006657 & 1.538479 \\
\hline\hline
\end{tabular}
\caption{Results for tandem queues at changing variance.}\label{B-CHA}
\end{center}
\end{table}

\bibliographystyle{utphys}

\bibliography{Algebraic_modelling_and_performance_evaluation_of_acyclic_fork-join_queueing_networks}

\end{document}